\documentclass{amsart}

\usepackage{subfiles}
\usepackage{comment}
\usepackage{float}
\usepackage{marginnote}
\usepackage{tabu}
\usepackage{euscript}
\usepackage[dvipsnames]{xcolor}   		             		
\usepackage{graphicx}			
\usepackage{amssymb}
\usepackage{mathrsfs}
\usepackage{amsthm}
\usepackage{amsmath}
\usepackage{stmaryrd}
\usepackage{tikz}
\usepackage{tikz-cd}
\usetikzlibrary{calc,arrows}
\usepackage{accents}
\usepackage{upgreek}
\usepackage{enumerate}
\usepackage{bm}
\usepackage{mathtools}
\usepackage[all]{xy}
\usepackage{caption}
\usepackage{url}
\usepackage{float}
\usepackage{todonotes} 
\usepackage{colonequals}
\usepackage{bbm}
\usepackage{longtable}
\usepackage[full]{textcomp}
\usepackage[cal=cm]{mathalfa}
\usepackage{xparse}
\usepackage{comment}
\usepackage[noadjust]{cite} 

\usepackage{faktor} 
\usepackage{xfrac} 

\usetikzlibrary{calc}
\usetikzlibrary{fadings}
\usetikzlibrary{decorations.pathmorphing}
\usetikzlibrary{decorations.pathreplacing}
\usepackage{tikz,tikz-cd,tikz-3dplot}
\usepackage{pgfplots}
\usetikzlibrary{arrows,shadows,positioning, calc, decorations.markings, 
hobby,quotes,angles,decorations.pathreplacing,intersections,shapes}
\usepgflibrary{shapes.geometric}
\usetikzlibrary{fillbetween,backgrounds}
\usetikzlibrary{arrows.meta} 

\usepackage[margin=1.05in]{geometry}

\tikzset{
  commutative diagrams/.cd, 
  arrow style=tikz, 
  diagrams={>=stealth}
}
\tikzset{
  arrow/.pic={\path[tips,every arrow/.try,->,>=#1] (0,0) -- +(0,4pt);},
  pics/arrow/.default={triangle 90}
}
\tikzset{->-/.style={decoration={
  markings,
  mark=at position .6 with {\arrow{latex}}},postaction={decorate}}
  }
\tikzset{
  c/.style={every coordinate/.try}
}

\theoremstyle{definition}
\newtheorem{innercustomthm}{Theorem}
\newenvironment{customthm}[1]
  {\renewcommand\theinnercustomthm{#1}\innercustomthm}
  {\endinnercustomthm}

\theoremstyle{definition}

  \theoremstyle{definition}

\makeatletter
\def\@tocline#1#2#3#4#5#6#7{\relax
  \ifnum #1>\c@tocdepth 
  \else
    \par \addpenalty\@secpenalty\addvspace{#2}%
    \begingroup \hyphenpenalty\@M
    \@ifempty{#4}{%
      \@tempdima\csname r@tocindent\number#1\endcsname\relax
    }{%
      \@tempdima#4\relax
    }%
    \parindent\z@ \leftskip#3\relax \advance\leftskip\@tempdima\relax
    \rightskip\@pnumwidth plus4em \parfillskip-\@pnumwidth
    #5\leavevmode\hskip-\@tempdima
      \ifcase #1
       \or\or \hskip 1em \or \hskip 2em \else \hskip 3em \fi%
      #6\nobreak\relax
    \dotfill\hbox to\@pnumwidth{\@tocpagenum{#7}}\par
    \nobreak
    \endgroup
  \fi}
\makeatother

\newcounter{marginnote}
\DeclareMathAlphabet{\mathpzc}{OT1}{pzc}{m}{it}

\usepackage[backref=page]{hyperref}
\hypersetup{
  colorlinks   = true,          
  urlcolor     = blue!65!black,          
  linkcolor    = blue!65!black,          
  citecolor   = blue             
}

\pgfplotsset{compat=1.18}

\theoremstyle{definition}
\newtheorem{theorem}{Theorem}[section]

\newtheorem*{claim*}{Claim}

\newtheorem{corollary}[theorem]{Corollary}

\newtheorem{lemma}[theorem]{Lemma}
\newtheorem{proposition}[theorem]{Proposition}
\newtheorem{remark}[theorem]{Remark}

\newtheorem*{runningexample*}{Running example}

\newtheorem*{aside*}{Aside}

\newtheorem{definition}[theorem]{Definition}

\newtheorem*{notation*}{Notation} 
\newtheorem{proposition-definition}[theorem]{Proposition-Definition}
\newtheorem{theorem-definition}[theorem]{Theorem-Definition} 

\newtheorem*{thm:motions_to_monodromies}{Theorem~\ref{thm:motions_to_monodromies}}
\newtheorem*{thm:monodromy_split}{Theorem~\ref{thm:monodromy_group_split_links}}
\newtheorem*{thm:hopf}{Theorem~\ref{thm:monodromy_group_hopf-unknots}}

\DeclareMathOperator{\id}{id}

\newcommand{\bcd}{\begin{center}\begin{tikzcd}}
\newcommand{\ecd}{\end{tikzcd}\end{center}}

\newcommand{\Z}{\mathbb{Z}}

\newcommand{\R}{\mathbb{R}}

\newcommand{\Dcal}{\mathcal{D}}

\newcommand{\calM}{\mathcal{M}}

\newcommand{\calP}{\mathcal{P}} 

\newcommand{\calE}{\mathcal{E}}

\newcommand{\motion}[2]{\mathcal{M}\left({#1}, {#2}\right)} 

\newcommand{\diff}{\operatorname{Diff}} 
\newcommand{\emb}{\operatorname{Emb}} 
\newcommand{\sub}{\operatorname{Sub}} 
\newcommand{\leg}

\DeclareMathOperator{\PuMon}{PuMon}
\DeclareMathOperator{\Kh}{Kh}
\DeclareMathOperator{\Mon}{Mon}
\DeclareMathOperator{\uMon}{uMon}
\DeclareMathOperator{\PMon}{PMon}
\DeclareMathOperator{\Aut}{Aut}
\DeclareMathOperator{\FR}{FR}

\begin{document}


\title{Khovanov monodromy groups via motions}

\author[Gabriel Corrigan]{Gabriel Corrigan}
\email{g.corrigan.1@research.gla.ac.uk}

\author[Livio Ferretti]{Livio Ferretti}
\email{livio.ferretti@glasgow.ac.uk}

\author[Isacco Nonino]{Isacco Nonino}
\email{i.nonino.1@research.gla.ac.uk}

\author[Susanna Terron]{Susanna Terron}
\email{s.terron.1@research.gla.ac.uk}


\thanks{This work was developed and written without the use of artificial intelligence models. The authors do not consent to the processing of this document by artificial intelligence models.
\newline \indent \textsc{School of Mathematics and Statistics, University of Glasgow, University Place, Glasgow, G12 8QQ}}

\begin{abstract}
    For any link, we define a monodromy map from the motion group of the link to the group of automorphisms of the link's Khovanov homology. The image of this map is the \emph{unoriented monodromy group} of the link. This map allows us to convert results concerning motion groups of links into ones about their Khovanov monodromy. In particular, we use a characterisation of the motion groups of split links to write their unoriented monodromy groups as an explicit semidirect product in terms of their unsplit pieces. We give some demonstrative examples, computing the unoriented monodromy groups of unlinks, Hopf links, and split links composed of pieces thereof.
\end{abstract}
\maketitle

\section{Introduction}

Khovanov homology is a powerful link invariant. To each oriented diagram~$D$ of a link in~$\R^3$, it associates in a purely combinatorial way a chain complex~$C(D)$ of free abelian groups, whose homology~$\Kh(D)$ is, up to isomorphism, an invariant of the link\footnote{Throughout the paper, we work with Khovanov homology with integer coefficients.}. Since its introduction in \cite{Khovanov_OG}, it has had remarkable applications in low-dimensional topology. One striking peculiarity of this theory is that its definition is purely diagrammatic, and the resulting invariant is indeed very useful for understanding combinatorial properties of links, yet it also detects strong topological information \cite{Rasmussen, Piccirillo, Hayden_Sundberg, Ren-Willis}. The key feature enabling extraction of topological information from such a combinatorial gadget is functoriality: oriented cobordisms in~$\R^3\times [0,1]$ between links induce morphisms on Khovanov homology. Such induced morphisms appear naturally in the theory and were defined by Khovanov himself, who moreover conjectured that, up to sign, they should be invariant under ambient isotopy of the link cobordism. This conjecture was later proved by Jacobsson, provided the ambient isotopy of the cobordism fixes the boundary setwise \cite{Jacobsson}.

Let us quickly recall the construction of the induced morphisms on homology. Since Khovanov homology is defined starting from link diagrams, to define the induced maps one needs a diagrammatic description of cobordisms. Given two oriented links~$L_0$ and~$L_1$ presented by oriented diagrams~$D_0$ and~$D_1$, a generic cobordism~$\Sigma$ from~$L_0$ to~$L_1$ can be presented as a finite sequence of diagrams, starting with~$D_0$ and ending with~$D_1$, where two successive diagrams in the sequence are related by a planar isotopy, a Reidemeister move, or a Morse move. Such a sequence of diagrams is called a \textit{movie}. Moreover, any two movies representing ambiently isotopic cobordisms can be related by a sequence of Carter--Saito \textit{movie moves} \cite{Carter-Saito-1}. To associate a morphism at the level of Khovanov homology to a cobordism, one thus needs to define maps induced by Reidemeister and Morse moves, and use movie moves to check that the morphisms induced by movies representing isotopic cobordisms agree up to sign. This is the main content of Jacobsson's paper \cite{Jacobsson}.
Given an oriented cobordism~$\Sigma$ as above, the associated morphism is well-defined up to sign, and we denote it by~$\Kh(\Sigma)$.

Crucially, the requirement that one only considers cobordisms up to ambient isotopy \emph{rel.~boundary} is necessary to get well-defined induced morphisms. The map induced by a trivial cobordism -- i.e.~a cobordism of the form~$L \times [0,1]$ -- is trivial. On the other hand, there are cobordisms that are ambiently isotopic to the trivial cylinder, \emph{albeit not rel. boundary}, which do not induce the trivial map on Khovanov homology \cite{Jacobsson}. This observation leads to considering the morphisms induced on Khovanov homology by those cobordisms whose boundary consists of two copies of the same link~$L$ and which are ambiently isotopic (not rel. boundary) to the trivial cobordism~$L \times [0,1]$. From the point of view of movies, these cobordisms correspond to so-called \textit{circular movies}.

Let~$D$ be an oriented link diagram. A \textit{circular movie} starring~$D$ is a movie whose first and last stills are both~$D$ and that does not involve any Morse move. For an oriented link~$L$ in~$\R^3$ represented by a diagram~$D$, the \emph{monodromy group} of~$D$ is the subgroup~$\Mon(D) \leq \Aut(\Kh(D))$ consisting of isomorphisms induced by circular movies from~$D$ to itself.

The monodromy group is a link invariant first defined by Jacobsson, who moreover proved that it is in general not trivial; for instance, the knot~$8_{18}$ has non-trivial monodromy \cite{Jacobsson}. The presence of non-trivial monodromy is a sometimes overlooked but crucial point of the theory, being a witness of the lack of naturality of Khovanov homology. In fact, Khovanov homology truly only associates to each link an \emph{isomorphism class} of groups, as opposed to the concrete group associated to a diagram. While this subtle distinction is irrelevant when considering Khovanov homology simply as a link invariant, it becomes important when working with functoriality.

Since its introduction, the monodromy group has been, to the best of our knowledge, mostly ignored. One of the goals of this paper is to reverse this trend and initiate a systematic study of monodromy groups. Our starting observation is that the notion of circular movie clearly reminds one of a \textit{motion of a link}. The group of motions of a submanifold~$N$ inside a manifold~$M$ was first introduced and studied in  \cite{Dahm62,Goldsmith}. Roughly speaking, a motion of~$N$ in~$M$ is a path of diffeomorphisms of~$M$ which begins at the identity and ends at a diffeomorphism which preserves~$N$ setwise, and the \textit{motion group}~$\mathcal{M}(N,M)$ is the group of relative homotopy classes of such motions. We will always care about motions of a link~$L$ in~$\R^3$, so we simply write~$\mathcal{M}(L) = \mathcal{M}(L, \R^3)$. We refer to Section~\ref{section:prelim_motion_groups} for precise definitions. Indeed, we will see that circular movies are just diagrammatic representations of motions; see Proposition~\ref{prop:motions_to_circular_movies}. The key idea is thus to connect the theory of monodromy groups with that of motion groups of links.

In order to be able to use results on unoriented motion groups (for which the theory is more developed than for oriented motions), we introduce the \textit{unoriented monodromy group}~$\uMon(D)$ of a diagram~$D$; see Section~\ref{section:motions_to_monodromies}. This is a group that contains Jacobsson's monodromy group as a subgroup, and consists of maps induced by cobordisms that, while preserving the orientation of the link, might change the orientations on the diagrams. Moreover, due to the sign ambiguity in the functoriality of~$\Kh$, we will work with~$\PuMon(D)$ rather than~$\uMon(D)$ itself. Here~$\PuMon(D)$ denotes the \emph{projective unoriented monodromy group}, i.e.~the quotient~$\uMon(D) / \{\pm1\}$. The connection we draw is the following structural theorem.

\begin{customthm}{A}\label{thm:motions_to_monodromies}
    There is a well-defined surjective group homomorphism~$$\mu \colon \mathcal{M}(L) \twoheadrightarrow \PuMon(L),$$
    defined as~$\mu \colon \gamma \mapsto \Kh(\Sigma_{\gamma})$, where~$\Sigma_{\gamma}$ is a generic cobordism associated to the motion~$\gamma$. We call this map the \emph{Khovanov motion picture map}.
    In particular,~$\mu$ induces the isomorphism~$$ \PuMon(L) \cong \faktor{\calM(L)}{\ker(\mu)}.$$
\end{customthm}

The value of such a connection is twofold. On the one hand, this provides an alternative to the Dahm homomorphism \cite{Dahm62}, which is induced by the action of motions on the fundamental group of the link complement. While we expect the Khovanov motion picture map to be a coarser invariant than the Dahm homomorphism, it has the advantage of being algorithmically computable and resulting in simple matrices. We therefore believe it should be of interest to the motion group community, as it may help in detecting nontrivial motions. 

On the other hand, by studying the kernel of~$\mu$ one can use knowledge of motion groups to compute monodromy groups. As an example of such computations, we apply results of Boyd-Bregman \cite{Boyd-Bregman} to study the unoriented monodromy group of split links. Below,~$\calP_L$ denotes the subgroup consisting of permutations of the link pieces which may only permute ambiently isotopic pieces; cf. Theorem~\ref{thm:motion_group_split_link} and the subsequent discussion.

\begin{customthm}{B}\label{thm:monodromy_group_split_links}
    Let~$L$ be a split link,~$L=\sqcup_iL_i$. The projective unoriented monodromy group of~$L$ is 
   ~$$ \PuMon(L)\cong \left(\prod _i\mu (\calM(L_i))\right)\rtimes \calP_L.$$
    In particular, we obtain that the unoriented monodromy group of a split link is a semidirect product of a permutation subgroup and a subgroup depending only on the actions of motions of single pieces. In the case where the Khovanov homology of the link pieces~$L_i$ is torsion-free, we obtain (Corollary \ref{cor:refined formula when there is no torsion}) the refined formula
   ~$$ \PuMon(L)\cong \left(\prod _i \PuMon(L_i)\right)\rtimes \calP_L.$$
\end{customthm}

As an application of Theorem \ref{thm:monodromy_group_split_links}, we completely determine the (projective) unoriented monodromy group of links in~$\mathbb{R}^3$ which consist of a split union of~$n$ unknots and~$m$ Hopf links. 
To the best of our knowledge, the monodromy group had so far been fully computed only for the unknot and~$2$-components unlink \cite{Collari}, and this is the first infinite family of links for which the unoriented monodromy group is determined. In the following,~$S_k$ denotes the symmetric group on~$k$ symbols.

\begin{customthm}{C}\label{thm:monodromy_group_hopf-unknots}
    Let~$L$ be a split link composed of~$n$ unknots and~$m$ Hopf links. Then we have 
   ~$$\PuMon(L)\cong \left(\Z_2^n \times (\Z_2\times \Z_2)^m\right) \rtimes (S_n\times S_m).$$
\end{customthm}

As a closing remark, note that we chose to work with motions of unoriented submanifolds, and thus define unoriented monodromy groups, as this seems the most natural setting from the motion group viewpoint. One can also define the \textit{oriented motion group}~$\mathcal{M}_{or}(L)$, and the proof of Theorem~\ref{thm:motions_to_monodromies} similarly yields an \emph{oriented Khovanov motion picture map}~$\mu_{or} \colon  \mathcal{M}_{or}(L) \twoheadrightarrow \PMon(L)$. Moreover, the two motion picture maps fit in the following commutative diagram.
\[\begin{tikzcd}
	{\mathcal{M}_{or}(L)} & {\mathcal{M}(L)} \\
	{\PMon(L)} & {\PuMon(L) } \\
	\arrow[two heads, from=1-1, to=2-1]
	\arrow[two heads, from=1-2, to=2-2]
	\arrow[hook, from=1-1, to=1-2]
    \arrow[hook, from=2-1, to=2-2]
\end{tikzcd}\]
 Therefore, our techniques can equally well be used to study the classical monodromy group.
 
\subsection*{Structure of the paper}
The paper is organised as follows. In Section~\ref{section:prelim_motion_groups} we recall some relevant facts about motion groups.  In Section~\ref{section:motions_to_monodromies} we introduce the unoriented monodromy group, study the relationship between motion groups and Khovanov homology, and prove Theorem \ref{thm:motions_to_monodromies}. In Section \ref{section:monodromy_split_links} we apply these techniques to prove Theorem \ref{thm:monodromy_group_split_links} and Theorem \ref{thm:monodromy_group_hopf-unknots}.

\subsection*{Acknowledgments}
The authors thank Philipp Bader, Carlo Collari, Naageswaran Manikandan, and Brendan Owens for useful conversations and insights, and Rachael Boyd for suggesting the application of \cite[Theorem A]{Boyd-Bregman}. GC, IN, and ST thank their doctoral supervisors for their support and guidance.

LF acknowledges support by the Swiss National Science Foundation Postdoc.Mobility fellowship 225434. IN was supported by EPSRC Studentship No. ~EP/W524359/1. ST is supported by the EPSRC [grant number EP/Y035232/1], Centre for Doctoral Training in Algebra, Geometry and Quantum Fields (AGQ).

\section{Preliminaries on motion groups}\label{section:prelim_motion_groups}

In this section, we define motion groups and recall some relevant theorems. For some additional background, see \cite{Goldsmith, Boyd-Bregman}. In what follows,~$M$ is a smooth manifold,~$N\subset M$ a smooth submanifold,~$\diff_c(M)$ denotes the group of compactly supported diffeomorphisms of~$M$, and~$\diff_c(M, N)$ is the subgroup of those diffeomorphisms preserving~$N$ setwise.

\begin{definition}\label{def:smooth motion}
    A path~$\gamma = \{\gamma_t\}_{t \in [0,1]}$ in~$\diff_c(M)$ such that~$\gamma_0 = \id_M$ and~$\gamma_1 \in \diff_c(M, N)$ is called a \emph{motion of $N$ in $M$}.  
\end{definition}


\begin{definition}\label{def:stationary_motion}
A motion~$\gamma$ is \emph{stationary} if~$\gamma_t \in \diff_c(M, N)$ for all~$t \in [0, 1]$.
\end{definition}

The \emph{product}~$\kappa \circ \gamma$ of two motions~$\gamma, \kappa$ of~$N$ in~$M$ is given by
\begin{equation*}
    (\kappa \circ \gamma)_t =
    \begin{cases}
        \gamma_{2t} & 0 \leq t \leq \frac{1}{2} \\
        \kappa_{2(t-\frac{1}{2})} \circ \gamma_1 & \frac{1}{2} \leq t \leq 1.
    \end{cases}
\end{equation*}

In other words: do~$\gamma$ twice as fast, and then do~$\kappa$ twice as fast translated by~$\gamma_1$. The \emph{inverse}~$\gamma^{-1}$ of a motion~$\gamma$ is given by~$(\gamma^{-1})_t = \gamma_{1-t} \circ \gamma_1^{-1}$.

\begin{definition}\label{def:equivalent_motions}
Two motions~$\gamma, \kappa$ are said to be \emph{equivalent} (written~$\gamma \equiv \kappa$) if~$\kappa^{-1} \circ \gamma$ is homotopic relative to the endpoints to a stationary motion. 
\end{definition}

The set of equivalence classes of motions forms a group with operation the product defined above. Note that it is simply the relative fundamental group~$\pi_1 \left( \diff_c(M), \diff_c(M, N), \id_M \right)$. We remark that although \emph{a priori} the relative fundamental group is simply a \emph{set} of homotopy classes of paths, this set can be given a group structure in two ways, since~$\diff_c$ is a topological group. One is the product just defined; the other is by pointwise composition~$(\kappa \cdot \gamma)_t = \kappa_t \circ \gamma_t$. One can show that the two paths~$\kappa \cdot \gamma$ and~$\kappa \circ \gamma$ are homotopic relative to the endpoints. 

\begin{definition}\label{def:smooth motion group}
    The \emph{(smooth) motion group of~$N$ in~$M$} is the relative fundamental group 
    \[\motion{N}{M} \coloneqq \pi_1 \left( \diff_c(M), \diff_c(M, N), \id_M \right).\]
\end{definition}

The equivalence relation on motions can be rephrased as follows \cite[Proposition 2.8]{Goldsmith}. Let~$\gamma, \kappa$ be motions of~$N$ in~$M$. Then~$\gamma \equiv \kappa$ if and only if~$\gamma$ is homotopic to a motion~$\gamma'$ of~$N$ such that~$\gamma'_t(N) = \kappa_t(N)$ for all~$t \in [0, 1]$.

Motion groups are helpfully understood as fundamental groups of embedding spaces. Let~$\emb(N, M)$ denote the space of smooth embeddings of a smooth manifold~$N$ into~$M$. Equipping~$\emb(N, M)$ with the~$C^\infty$-Whitney topology makes it a principal~$\diff(N)$-bundle, where~$\diff(N)$ acts by precomposition. The quotient 
\[\sub(N, M) \coloneqq \faktor{\emb(N, M)}{\diff(N)}\]
is the \emph{unparametrised embedding space}. 

Fix a favoured embedding~$i \colon N \hookrightarrow M$ to be a chosen basepoint of~$\emb(N, M)$, and its image~$i(N)$ to be a chosen basepoint of~$\sub(N, M)$. Now, let~$\calE(i)$ denote the connected component of~$\sub(N, M)$ which contains the basepoint~$i(N)$. Wattenberg \cite{WattenbergSmoothMotionsUnlink72} showed that the motion group is in fact isomorphic to the fundamental group of the unparametrised embedding space.

\begin{theorem}[\cite{WattenbergSmoothMotionsUnlink72}, Lemma 1.4]\label{thm:smooth motion group is pi_1 of embedding space Wattenberg}
   ~$\motion{i(N)}{M} \cong \pi_1\left(\calE(i), i(N)\right)$.
\end{theorem}

In this paper, we are interested in motion groups of links in~$\R^3$. Since the ambient space will always be~$\R^3$, we drop it from the notation and simply write~$\calM(L) \coloneqq \calM(L,\R^3)$. Such groups are hard to compute in general, but have been explicitly studied in a few cases, as we now review. These computations will be instrumental in the proofs of our main theorems. 

The motion group of the unknot was computed by Dahm in his thesis to be~$\mathbb{Z}_2$ \cite{Dahm62, Goldsmith, BrendleHatcherRingsWickets13}. [In fact, he computed the motion group of the~$n$-component unlink to be isomorphic to the symmetric automorphism group of the free group~$F_n$, also known as the \emph{loop braid group}.] The motion group of the Hopf link is computed to be the quaternionic group~$Q_8$ \cite{DamianiKamadaRingHTrivialLinks19, Boyd-Bregman-Hopf}.

Let~$L$ be a \emph{split link} in~$\R^3$. Recall that this means that~$L$ can be written as a disjoint union~$L = L_1 \sqcup \cdots\sqcup L_n$, where the \emph{pieces}~$L_i \neq \emptyset$ are contained in pairwise disjoint balls~$B_i$. We identify~$L$ with the image of a preferred embedding~$i\colon \bigsqcup S^1 \rightarrow \R^3$, which we use as the basepoint of the embedding space. Boyd--Bregman \cite{Boyd-Bregman} give an explicit description of the motion group of a split link in~$\mathbb{R}^3$.

\begin{theorem}{\cite[Theorem A]{Boyd-Bregman}}\label{thm:motion_group_split_link}
    The  motion group~$\calM(L)$ of a split link~$L\subset\R^3$ is isomorphic to~$$\left(\FR(L) \rtimes \left(\prod_i\calM(L_i) \right)\right)\rtimes \calP_{L}.$$
\end{theorem}

We spend some words to identify the groups appearing in the statement.

\begin{itemize}
    \item~$\calP_{L}$: there is a natural action of the symmetric group~$S_n$ on the set of the~$n$ pieces by permutations, and~$\calP_{L}$ is the subgroup of permutations which only permute link pieces when they are ambiently isotopic.
    
    \item~$\prod_i\calM(L_i)$: this is the direct product of the motion groups of the single link pieces of the split link.
    
    \item~$\FR(L)$: this is a subgroup of~$\pi_1\left(\R^3\setminus L\right)$ generated by a certain set of conjugations. Precisely, if ~$H_i=\pi_1(\R^3 \setminus L_i)$, then~$\pi_1\left(\R^3 \setminus L\right) \cong H_1 \ast \cdots \ast H_n$, and~$\FR(L)$ is generated by conjugations of a factor~$H_i$ by an element~$g \in H_j$, where~$i\neq j$. This is called the \emph{Fouxe-Rabinovitch} group of~$L$ after \cite{FouxeRabinovitchAutsFreeProductsI40, FouxeRabinovitchAutsFreeProductsII41}. The geometric description of such motions is quite simple: these correspond to taking the link piece~$L_i$ and \emph{looping it inside} the link piece~$L_j$ along the loop~$g$ of~$H_j$, as depicted in Figure \ref{figure:conjugation_motion}. 
\end{itemize}

\begin{figure}[htb]
\centering
\begin{tikzpicture}
\node[anchor=south west,inner sep=0] at (0,0){\includegraphics[width=10cm]{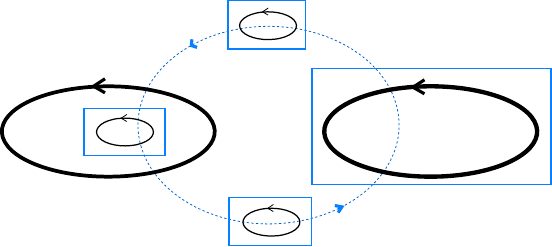}};
\end{tikzpicture}
\caption{A picture of the conjugation motion for the 2 component unlink. The rightmost component is looping inside the leftmost.}
\label{figure:conjugation_motion}
\end{figure}

\section{(Projective) unoriented monodromy groups via motion groups}\label{section:motions_to_monodromies}

We will now study the relationship between motion groups and monodromy groups. Throughout this section and the rest of the paper the links we are considering are always embedded in~$\R^3$, and we always fix a preferred projection~$p \colon \R^3 \to \R^2$. We choose the basepoint for the motion group to be the image of an embedding such that its projection via~$p$ is a link diagram~$D$.

\subsection{The unoriented monodromy group}\label{subsec:unoriented monodromy group}

As discussed in the introduction, we need to introduce an \emph{unoriented monodromy group}. The following definition is key.

\begin{definition}\label{def:diagram_collection}
    Let~$L$ be the image of an oriented link embedding in~$\R^3$ and~$D$ an oriented diagram for~$L$. Consider the set~$\Dcal=\{D_1, \dots, D_k\}$ of all oriented diagrams for~$L$ having underlying unoriented diagrams coinciding with that of~$D$. We call this set the \emph{diagram collection} of~$D$. 
\end{definition}

Note that the Khovanov chain complex associated to each diagram~$D_i\in\Dcal$ is exactly the same. Let us denote it by~$C$, and let~$\Aut(H_*(C))$ be the group of automorphisms of the corresponding Khovanov homology.

\begin{definition}\label{def: unoriented_circ_movie}
    Let~$D$ be an oriented link diagram and~$\Dcal$ its diagram collection. We define an \textit{unoriented circular movie} of~$\Dcal$ to be a movie with only Reidemeister moves and planar isotopies whose first still is~$D_i$ and last still is~$D_j$, with~$D_i, D_j \in \Dcal$. 
\end{definition}

We can now define the unoriented monodromy group.

\begin{definition}\label{def: unoriented_monodromy}
Let~$L$ be an oriented link and~$D$ a fixed oriented diagram for~$L$. The \emph{unoriented monodromy group} of~$D$ is the subgroup of~$\Aut(H_*(C))$ given by the automorphisms induced by unoriented circular movies of~$\Dcal$. We will denote it by~$\uMon(D)$. We denote by~$\PuMon(D)$ the quotient~$\uMon(D) / \{\pm1\}$.
\end{definition}

In fact, these groups are invariants of the link. We will therefore use the notation~$\uMon(L)$ and~$\PuMon(L)$ without further comment.

\begin{remark}\label{rem:operation}
    We elaborate why this is actually a well-defined subgroup. At first glance, unoriented circular movies are not always composable. This necessitates our careful setup, where composition of unoriented circular movies is as we are about to describe; this will respect composition of automorphisms. Precisely, replacing the first still~$D_i$ of an unoriented circular movie with another~$D_j$ in the same diagram collection gives a different unoriented circular movie which nonetheless induces the same automorphism on~$H_\ast(C)$. Therefore, there is always a choice of orientation on initial stills (which we are free to make) which makes two unoriented circular movies composable. This construction respects composition in~$\Aut(H_\ast(C))$.
\end{remark}

\subsection{The Khovanov motion picture map}\label{subsec:the Khovanov motion picture map}

The first important thing to notice is that there is a deep connection between motions and unoriented circular movies. Our first step is to associate a cobordism to each motion.

\begin{lemma}\label{lem:homotopic_ambient_isotopy}
    To every motion~$\gamma \in\calM(L)$ we can associate a cobordism~$\Sigma_\gamma$, well-defined up to ambient isotopy relative to the boundary.
\end{lemma}

\begin{proof}
    Given a path~$\gamma = \{\gamma_s\}_{s\in[0,1]}$ in~$\diff_c(\R^3)$ representing a motion in~$\calM(L)$, let us denote by~$\Sigma_\gamma$ the trace~$(\gamma_s(L),s)$ in~$\R^3 \times [0,1]$. Then~$\Sigma_\gamma$ defines a cobordism from~$L$ to~$\gamma_1(L)=L$.
    If~$\gamma$ and~$\kappa$ are two representatives of the same element of~$\calM(L)$, then there exists a relative homotopy~$H$  between them, i.e.~\begin{align*}
        H \colon [0,1]_t \times [0,1]_s & \to \diff_c (\mathbb{R}^3)\\
         (t,s)  & \mapsto H_t(s)
    \end{align*}
    satisfying
    \begin{align*}
        H_0(s)&=\gamma_s\\
        H_1(s)&=\kappa_s\\
        H_t(0)&=\id_{\mathbb{R}^3}\\
        H_t(1)&\in \diff_c(\mathbb{R}^3,L)
    \end{align*}
    for all~$s, t\in[0,1]$.
    
    Consider the map~$$H' \colon L \times  [0,1]_t\times [0,1]_s \to \mathbb{R}^3 \times [0,1]_s$$
    given by 
    ~$$ H'_t(L,s) = (H_t(s)(L),s).$$
    The map~$H'$ is an isotopy between the embeddings of~$L \times [0,1]_s \hookrightarrow \mathbb{R}^3 \times [0,1]$ given by~$H'_0$ and~$H'_1$, which correspond to the cobordisms~$\Sigma_\gamma=(\gamma_s(L),s)$ and~$\Sigma_{\kappa}=(\kappa_s(L),s)$ traced by the two paths.
    Note that the boundary of the cobordisms is fixed throughout the whole isotopy:~$$H'_t(L,0)= (H_t(0)(L),0)=(L,0),$$
    and~$$ H'_t(L,1)=(H_t(1)(L),1)=(L,1).$$
    Using the isotopy extension theorem, we obtain a map~$$ \overline{H} \colon \mathbb{R}^3 \times [0,1]_t\times [0,1]_s \to \mathbb{R}^3 \times [0,1]_s$$ which is an ambient isotopy between the two surfaces and fixes the boundary of the cobordisms.     
\end{proof}

Such cobordisms are actually represented by unoriented circular movies, as we now prove.

\begin{proposition}\label{prop:motions_to_circular_movies}
    Let~$\gamma$ be a smooth motion of an oriented link~$L$. Then the cobordism~$\Sigma_{\gamma}$ is represented by an unoriented circular movie.
\end{proposition}

\begin{proof}
    The image~$p(L)$ under the fixed preferred projection is an oriented diagram~$D$ of~$L$.
    The trace~$\Sigma_\gamma$ is a cobordism between~$L = \gamma_0(L)$ and~$\gamma_1(L)$. Note that, by definition of motion, the link~$\gamma_1(L)$ coincides with~$L$ as an unoriented submanifold of~$\R^3$, but its orientation may have changed throughout the motion. In other words,~$p\circ\gamma_1(L)$ determines an oriented diagram~$D'$ which belongs to the diagram collection of~$D$. Moreover,~$\Sigma_\gamma$ is a generic surface, and is therefore represented by a movie from~$D$ to~$D'$. The motion~$\gamma$ determines for each time~$t$ an embedded submanifold~$\gamma_t(L)$, so there cannot be any Morse modification in the movie of~$\Sigma_\gamma$. In total, we have constructed an unoriented circular movie associated to~$\gamma$.
\end{proof}

We are now ready to discuss the Khovanov motion picture map.

\begin{thm:motions_to_monodromies}
    There is a well-defined surjective group homomorphism~$$\mu \colon \mathcal{M}(L) \twoheadrightarrow \PuMon(L),$$
    defined as~$\mu \colon \gamma \mapsto \Kh(\Sigma_{\gamma})$. We call this map the \emph{Khovanov motion picture map}.
    In particular,~$\mu$ induces the isomorphism~$$ \PuMon(L) \cong \faktor{\calM(L)}{\ker(\mu)}.$$
\end{thm:motions_to_monodromies}

\begin{proof}
Let~$D$ be an oriented diagram for~$L$. To define~$\mu$, we wish to associate to a motion~$\gamma$ the morphism on Khovanov homology induced by the cobordism~$\Sigma_\gamma$ constructed in Lemma \ref{lem:homotopic_ambient_isotopy}. In order to do this, we ought to give~$\Sigma_\gamma$ an orientation. Motions trace out paths of unoriented submanifolds, so we have to make a choice. We give~$\Sigma_\gamma$ the orientation induced by having~$D$ as its oriented boundary at time~$s=0$. By Lemma \ref{lem:homotopic_ambient_isotopy} the map~$\mu$ is well-defined, and Proposition \ref{prop:motions_to_circular_movies} implies that the image is indeed contained in~$\PuMon(L)$.
 
We now show that~$\mu$ is a group homomorphism. We first need to rephrase the content of Remark \ref{rem:operation} in terms of cobordisms. Given another oriented link diagram~$D'$ in the diagram collection~$\Dcal$, the cobordism~$\Sigma_\gamma$ with the orientation induced by~$D'$ gives the same map on~$H_*(C)$ as the cobordism~$\Sigma_\gamma$ with the orientation induced by~$D$. This has the upshot that when we compose two motions~$\gamma, \kappa \in \calM(L)$, the resulting cobordism~$\Sigma_{\gamma \circ \kappa}$ induces the same element of~$\Aut(H_\ast(C))$ as the composition of those induced by~$\Sigma_\gamma$ and~$\Sigma_\kappa$:
\[\Kh\left(\Sigma_{\gamma \circ \kappa}\right) = \Kh\left(\Sigma_\gamma\right) \circ \Kh\left(\Sigma_\kappa\right).\]

In other words, we have~$\mu(\gamma \circ \kappa)=\mu(\gamma)\circ \mu(\kappa)$, so~$\mu$ is a group homomorphism.
 
All that remains is to prove surjectivity of~$\mu$. By definition, an element of~$\PuMon(L)$ is induced by an unoriented circular movie. In particular, any unoriented circular movie corresponds to an oriented cobordism~$\Sigma \subset \R^3 \times [0,1]$ which has no Morse modifications. This means that at every time~$t \in [0, 1]$, the section~$\Sigma_t \subset \R^3 \times \{t\}$ is an embedding of~$L$ in~$\R^3$, so the cobordism traces a path in the space of embeddings of~$L$. By isotopy extension this gives rise to a motion~$\gamma$ for which~$\Sigma_\gamma=\Sigma$. Hence, the image of~$\gamma$ under~$\mu$ is the automorphism~$\Kh(\Sigma)$, as required. 

The isomorphism~$\PuMon(L) \cong \calM(L)/\ker(\mu)$ is a direct consequence.
\end{proof}

Theorem \ref{thm:motions_to_monodromies} can be summarised as asserting the existence of the short exact sequence 
\begin{equation*}
        1 \to \ker (\mu) \to \calM(L) \to \PuMon(L)\to 1.
\end{equation*} 
In other words, one can understand the projective unoriented monodromy group by studying the kernel of the map~$\mu$ and the motion group of a link. Roughly speaking, this turns the problem into understanding generators of motion groups and their image under~$\mu$ instead of directly studying the algebraic setting of~$\Aut(\Kh(L))$. As recalled in Section~\ref{section:prelim_motion_groups}, the motion group of links has been well studied in certain cases, and we have explicit descriptions of the generators \cite{Dahm62, Goldsmith, Goldsmith_torus, BrendleHatcherRingsWickets13, BellingeriBodinBraidGroupNecklace16, DamianiKamadaRingHTrivialLinks19, Boyd-Bregman-Hopf, Boyd-Bregman}. In Section \ref{section:monodromy_split_links} we show, for certain links, which generators live in the kernel of~$\mu$ and compute the quotient~$\calM (L)/\ker (\mu) \cong \PuMon(L)$.

\section{Monodromy groups of split links}\label{section:monodromy_split_links}

In this section we prove Theorem \ref{thm:monodromy_group_split_links}, and use it to compute the monodromy group of split links composed of unknots and Hopf links.

\subsection{Monodromy groups of split links in general}\label{subsec:monodromy groups of split links in general}
 
Recall first the computation of the motion group of a split link~$L$ given in Theorem \ref{thm:motion_group_split_link} and the subsequent discussion of the pieces appearing in the statement. We start by proving some preparatory lemmas. Henceforth,~$L$ is a split link.

\begin{lemma}\label{lem:fr_kernel}
    The subgroup~$\FR(L)$ is in the kernel of the map~$\mu$.
\end{lemma}

\begin{proof}
     The result follows from \cite[Theorem 1.2]{Gujral-Levine} on \emph{partition-preserving cobordisms} bewteen split links, stating that two partition-preserving cobordisms induce the same map on Khovanov homology, provided their parts are isotopic rel. boundary. That is, the induced maps on Khovanov homology do not detect the linking between the components of such cobordisms. Now, the key observation is the following: every cobordism associated to a motion in~$\FR(L)$ preserves the partition of the split link and, moreover, every piece is ambiently isotopic rel. boundary to a trivial cobordism. In particular the induced map on Khovanov homology is trivial (up to sign). This concludes the proof.
\end{proof}
 
 \begin{lemma} \label{lem:intersection_symmetric}
    ~$\calP_L \cap \ker (\mu)~$ is trivial.
 \end{lemma}
 \begin{proof}
     For a split link~$L = L_1 \sqcup \cdots \sqcup L_n$, the Khovanov chain complex is the tensor product of the chain complexes of the split pieces, and by the K\"unneth formula,~$\Kh(L)$ contains a subgroup~$\Kh(L_1) \otimes \cdots \otimes \Kh(L_n)$, which is natural. The cobordisms induced by elements of~$\calP_L$ act on this subgroup by permutation. In particular, a non-trivial element of~$\calP_L$ induces a non-trivial automorphism of~$\Kh(L)$.
 \end{proof}

We will also need the following algebraic lemmas.

 \begin{lemma}\label{lem:reduction_to_tensors}
     Let~$G_1$ and~$G_2$ be finitely generated~$\Z$-modules with strictly positive rank. Let~$g_1 \colon G_1\rightarrow G_1$ and~$g_2 \colon G_2\rightarrow G_2$ be automorphisms such that~$g_1 \otimes g_2 = \varepsilon \id \colon G_1 \otimes G_2 \rightarrow G_1 \otimes G_2$, with~$\varepsilon\in \{\pm 1\}$. Then~$g_i = \varepsilon_i \id$ with~$\varepsilon_i \in \{\pm 1\}$ and~$\varepsilon_1 \varepsilon_2 = \varepsilon$.
 \end{lemma}

 \begin{proof}
     Write~$G_i = F_i\oplus T_i$ with the~$F_i$ torsion-free and the~$T_i$ torsion. Then~$g_i$ can be written as~$$g_i(x_i,y_i) = (\overline{g_i}(x_i),\tau_i(x_i)+\sigma_i(y_i)),$$ with~$\overline{g_i}:F_i\to F_i$ automorphisms of the free summands. Recall the standard decomposition~$$G_1\otimes G_2 = (F_1\otimes F_2) \oplus (F_1\otimes T_2)\oplus (T_1\otimes F_2)\oplus (T_1\otimes T_2).$$

     We first show that~$\overline{g_i} = \varepsilon_i \id$ for~$i = 1, 2$ with~$\varepsilon_1 \varepsilon_2=\varepsilon$. The induced map~$\overline{g_1}\otimes \overline{g_2} \colon F_1\otimes F_2 \to F_1\otimes F_2$ is an automorphism and since~$g_1\otimes g_2 = \varepsilon \id$ we deduce that~$\overline{g_1}\otimes \overline{g_2} = \varepsilon \id$. Representing~$\overline{g_i}$ as an integer matrix~$A^i$ and acting on basis elements of~$F_1\otimes F_2$, one deduces~$A^1_{a, b}A^2_{c, d} = \varepsilon\delta_{a,b}\delta_{c, d}$ for all indices~$a, b, c, d$. Varying over the indices then gives the desired result.

     Going back to the full tensor product, for~$x_1\otimes x_2\in F_1\otimes F_2$ we know by assumption that
    ~$$ (g_1\otimes g_2)(x_1\otimes x_2) = (\varepsilon_1x_1\otimes\varepsilon_2 x_2) + (\varepsilon_1x_1\otimes \tau_2(x_2)) + (\tau_1(x_1)\otimes\varepsilon_2x_2) + (\tau_1(x_1)\otimes\tau_2(x_2)) = \varepsilon x_1\otimes x_2.$$ Since all the terms of this sum live in different direct summands of~$G_1\otimes G_2$, it follows that~$\tau_1 = \tau_2 = 0$.

     Finally, for~$x_1\otimes y_2\in F_1\otimes T_2$ we have
    ~$$(g_1\otimes g_2)(x_1\otimes y_2) = \varepsilon_1 x_1 \otimes \sigma_2(y_2) = \varepsilon x_1\otimes y_2,$$
     which implies that~$\sigma_2 = \varepsilon_2 \id$. Similarly, by looking at the action on~$T_1\otimes F_2$ we deduce that~$\sigma_1 = \varepsilon_1 \id$. Altogether, we thus obtain that~$g_i = \varepsilon_i \id$ and with~$\varepsilon_i \in \{\pm 1\}$ and~$\varepsilon_1 \varepsilon_2 = \varepsilon$, as required.
 \end{proof}

 \begin{lemma}\label{lem:kunneth}
Let~$(C_1,d_1)$,~$(C_2,d_2)$ be chain complexes of finitely generated free~$\Z$-modules,~$f_i, g_i:C_i\rightarrow C_i$ chain maps for~$i\in\{1,2\}$ such that 
\begin{enumerate}[(i)]
    \item~$(\id \otimes \; \varepsilon_2 g_2)$ and~$(\varepsilon_1 g_1 \otimes \id)$ induce isomorphisms on the homology~$H_\ast(C_1 \otimes C_2)$; \label{enumitem:kunneth hyp 1}
    \item~$(f_i)_* = \varepsilon_i (g_i)_* \colon H_*(C_i)\rightarrow H_*(C_i)$, with~$\varepsilon_i \in \{\pm 1\}$; \label{enumitem:kunneth hyp 2}
    \item~$(f_1\otimes f_2)_* = \varepsilon_1\varepsilon_2 (g_1\otimes g_2)_* \colon H_*(C_1\otimes C_2) \rightarrow H_*(C_1\otimes C_2)$. \label{enumitem:kunneth hyp 3}
\end{enumerate}

Then~$(f_1\otimes \id)_* = \varepsilon_1 (g_1\otimes \id)_*$ and~$(\id\otimes f_2)_* = \varepsilon_2 (\id\otimes g_2)_*$ on~$H_*(C_1\otimes C_2)$.
     
 \end{lemma}

 \begin{proof}
     Consider the chain maps~$F_i = f_i -\varepsilon_i g_i \colon C_i\rightarrow C_i$. By hypothesis (\ref{enumitem:kunneth hyp 2}),~$(F_i)_* = 0$ on~$H_*(C_i)$. We want to show that~$(F_1\otimes \id)_* = (\id\otimes F_2)_* = 0$ on~$H_*(C_1\otimes C_2)$. Consider the K\"unneth exact sequence
     \begin{equation*}
         1 \to H_*(C_1)\otimes H_*(C_2) \to H_*(C_1\otimes C_2) \to \mathrm{Tor} (H_*(C_1),H_*(C_2)) \to 1.
     \end{equation*}

     \noindent In particular,~$H_*(C_1)\otimes H_*(C_2)$ is a natural subgroup of~$H_*(C_1\otimes C_2)$.
     
     First of all, for any chain map~$h:C_2\rightarrow C_2$, let us show that~$(F_1\otimes h)_*$ acts trivially on~$H_*(C_1)\otimes H_*(C_2)$ and maps~$H_*(C_1\otimes C_2)$ into this subgroup. To prove the first claim, just notice that by naturality of the K\"unneth exact sequence,
    ~$$(F_1\otimes h)_* |_ {H_*(C_1)\otimes H_*(C_2)} = (F_1)_* \otimes h_* = 0.$$
     For the second claim, we use the fact that K\"unneth's sequence splits (albeit non-canonically). More concretely, by \cite[Proposition V.10.6]{MacLane} a lift of~$\mathrm{Tor} (H_*(C_1),H_*(C_2))$ in~$H_*(C_1\otimes C_2)$ can be constructed as follows: given classes~$[\alpha_1] \in H_p(C_1)$ and~$[\alpha_2] \in H_q(C_2)$ of finite order, an integer~$n$ such that~$n[\alpha_i] = 0$, and chains~$\beta_i$ such that~$n\alpha_i = d_i \beta_i$, the chain~$z = \alpha_1\otimes \beta_2 - (-1)^p \beta_1\otimes\alpha_2$ is a cycle in~$C_1\otimes C_2$ lifting a class in~$\mathrm{Tor}(H_p(C_1),H_q(C_2))$, and every class in~$\mathrm{Tor}$ admits a lift of this form.

     Since~$(F_1)_* = 0$, we deduce that~$F_1(\alpha_1) = d_1\gamma$ for some chain~$\gamma\in C_1$, and
    ~$$(F_1\otimes h)(z) = d_1\gamma\otimes h(\beta_2) - (-1)^p F_1(\beta_1)\otimes h(\alpha_2) = d(\gamma\otimes h(\beta_2)) + (-1)^p (n\gamma - F_1(\beta_1))\otimes h(\alpha_2),$$
    so that
    ~$$(F_1\otimes h)_* [z] = (-1)^p [w\otimes h(\alpha_2)],$$
     with~$w = n\gamma -F_1(\beta_1) \in C_1$ a cycle of degree~$p+1$.
     Therefore,~$(F_1\otimes h)_* [z] \in H_{p+1}(C_1)\otimes H_q(C_2)$. This shows that~$(F_1\otimes h)_* (H_*(C_1\otimes C_2)) \subseteq H_*(C_1)\otimes H_*(C_2)~$, as required.

     The same argument also shows that, for any chain map~$l:C_1\rightarrow C_1$, the induced~$(l\otimes F_2)_*$ acts trivially on~$H_*(C_1)\otimes H_*(C_2)$ and maps~$H_*(C_1\otimes C_2)$ into this subgroup. However, crucially, while a map of the form~$F_1\otimes h$ maps (a lift of)~$\mathrm{Tor}(H_p(C_1),H_q(C_2))$ into~$H_{p+1}(C_1)\otimes H_q(C_2)$, a map of the form~$l\otimes F_2$ maps~$\mathrm{Tor}(H_p(C_1),H_q(C_2))$ into~$H_p(C_1)\otimes H_{q+1}(C_2)$.

     Finally, note that 
    ~$$F_1\otimes F_2 + F_1\otimes \varepsilon_2 g_2 + \varepsilon_1 g_1\otimes F_2 = f_1\otimes f_2 - \varepsilon_1 \varepsilon_2 g_1\otimes g_2,$$
     acts trivially on homology by hypothesis (\ref{enumitem:kunneth hyp 3}), so that
    ~$$(F_1\otimes F_2)_* + (F_1\otimes \varepsilon_2 g_2)_* + (\varepsilon_1 g_1\otimes F_2)_* = 0 .~$$
     Now,~$(F_1\otimes F_2)_* = (F_1\otimes \id)_* \circ (\id\otimes F_2)_* = 0$ since~$(\id\otimes F_2)_*$ has image in~$H_*(C_1)\otimes H_*(C_2)$ and~$(F_1\otimes \id)_*$ vanishes on this subgroup. We thus deduce that
    ~$$(F_1\otimes \varepsilon_2 g_2)_* + (\varepsilon_1 g_1\otimes F_2)_* = 0.$$ As remarked above, both~$(F_1\otimes \varepsilon_2 g_2)_*$ and~$(\varepsilon_1 g_1\otimes F_2)_*$ vanish on~$H_*(C_1)\otimes H_*(C_2)$, and they map lifts of~$\mathrm{Tor}(H_p(C_1),H_q(C_2))$ into summands of different bi-degree. It follows that 
    ~$$(F_1\otimes \varepsilon_2 g_2)_* = (\varepsilon_1 g_1\otimes F_2)_* = 0.$$
     By hypothesis (\ref{enumitem:kunneth hyp 1}),~$(\id\otimes \varepsilon_2 g_2)_*$ and~$(\varepsilon_1 g_1\otimes \id)_*$ are automorphisms, so we conclude that
    ~$$(F_1\otimes \id)_* = (\id\otimes F_2)_* = 0.$$    
 \end{proof}

 We are now ready to prove Theorem \ref{thm:monodromy_group_split_links}.
 
\begin{thm:monodromy_split}
   Let~$L=\sqcup_i^nL_i$ be a split link. The projective monodromy group of~$L$ is~$$ \PuMon(L)\cong \left(\prod _i^n \mu(\calM(L_i))\right)\rtimes \calP_L.$$
\end{thm:monodromy_split}
\begin{proof}
    Recall that by Theorem \ref{thm:motion_group_split_link},~$$
    \calM(L) = \left(\FR(L) \rtimes \prod_i^n \calM(L_i)\right) \rtimes \calP_L,$$ 
    where the Fouxe-Rabinovitch group $\FR(L)$ corresponds to conjugations, the middle factor is the direct product of the motion groups of the pieces, and the symmetric group factor $\calP_L$ permutes ambiently isotopic pieces. By Lemma \ref{lem:fr_kernel},~$\FR(L) \subseteq \ker(\mu)$, so ~$\mu$ factors as
    \[\begin{tikzcd}
	{\left(\FR(L) \rtimes \prod_i^n \calM(L_i)\right) \rtimes \calP_L} & \\
	& { \left(\prod_i^n \calM(L_i)\right) \rtimes \calP_L} \\
	{\PuMon(L)}
	\arrow[two heads, from=1-1, to=2-2]
	\arrow["\mu"', two heads, from=1-1, to=3-1]
	\arrow["{\mu'}"', two heads, from=2-2, to=3-1]
\end{tikzcd}\]

Let us first show that~$\mu'(\prod_i^n\calM(L_i))= \prod_i^n(\mu'(\calM (L_i)))$.  Since~$\mu'$ is a homomorphism, each of the subgroups~$\mu'(\calM (L_i))$ is normal in~$\mu'(\prod_i^n\calM(L_i))$, and every element of~$\mu'(\prod_i^n\calM(L_i))$ can be decomposed as a product of elements of these subgroups. We just need to prove that such a decomposition is unique. Consider motions~$\gamma_i, \gamma_i ' \in \calM(L_i)$ for all~$i=1,\dots, n$ such that ~$\mu'(\gamma_1,\cdots,\gamma_n) = \mu'(\gamma_1',\cdots,\gamma_n')$. 

For each~$i$, we take an unoriented circular movie representing~$\gamma_i$, and by abuse of notation we denote by~$\mu'(\gamma_i)$ and~$\mu'(\gamma_1, \dots, \gamma_n)$ the representatives of the corresponding projective classes induced by these choices of movies. 

The Khovanov chain complex of a split link is the tensor product of the chain complexes of the split pieces, and at the chain level the map induced by the tuple~$(\gamma_1,\cdots,\gamma_n)$ is the tensor product of the maps induced by each~$\gamma_i$ (and similarly for the~$\gamma_i'$). By the K\"unneth formula,~$\Kh(L)$ contains a natural subgroup~$\Kh(L_1)\otimes\cdots\otimes \Kh(L_n)$ on which~$\mu'(\gamma_1,\cdots,\gamma_n)$ acts as~$\mu'(\gamma_1)\otimes\cdots\otimes\mu'(\gamma_n)$. We thus deduce that the two maps~$\mu'(\gamma_1)\otimes\cdots\otimes\mu'(\gamma_n)$ and~$\mu'(\gamma_1')\otimes\cdots\otimes\mu'(\gamma_n')$ on~$\Kh(L_1)\otimes\cdots\otimes \Kh(L_n)$ agree up to a sign. 

By repeated applications of Lemma \ref{lem:reduction_to_tensors}, since Khovanov homology of a link always contains a non-trivial free summand, it follows that~$\mu'(\gamma_i) = \pm\mu'(\gamma_i')$ on~$\Kh(L_i)$. Now, by repeated applications of Lemma \ref{lem:kunneth} we deduce that~$(1, \cdots, 1, \gamma_i, 1, \cdots, 1)$ and~$(1, \cdots, 1, \gamma_i', 1, \cdots, 1)$ induce the same map up to sign on~$\Kh(L)$, so that 
$$\mu'(1, \cdots, 1, \gamma_i, 1, \cdots, 1) = \pm \mu'(1, \cdots, 1, \gamma_i', 1, \cdots, 1)$$ 
for all~$i$. Thus, every element of~$\mu'(\prod_i^n\calM(L_i))$ can be decomposed uniquely as a product of elements of the normal subgroups~$\mu'\calM (L_i))$, which implies~$$\mu'\left(\prod_i^n\calM(L_i)\right)= \prod_i^n\mu'(\calM (L_i)),$$ 
as claimed.

Since~$\mu'$ is a surjective homomorphism, we have that 
\[\PuMon(L) = \mu'\left(\prod_i^n\calM(L_i)\right)\cdot \mu'(\calP_L),\] with~$\mu'(\prod_i^n\calM(L_i))$ a normal subgroup. Moreover, elements in~$\mu'(\calP_L)$ act on~$\Kh(L_1)\otimes\cdots\otimes \Kh(L_n)$ by permuting the tensor components, while elements in~$\mu'(\prod_i^n\calM(L_i))$ preserve the tensor components. Therefore,~$\mu'(\calP_L) \cap \mu'(\prod_i^n\calM(L_i))$ is trivial, which implies that~$$\PuMon(L) = \mu'\left(\prod_i^n\calM(L_i)\right)\rtimes \mu'(\calP_L).$$

Finally, by Lemma \ref{lem:intersection_symmetric},~$\mu'$ restricts to an isomorphism on the subgroup~$\calP_L$, so that~$\mu'(\calP_L)\cong \calP_L$, while as shown above we have~$\mu'(\prod_i^n\calM(L_i))= \prod_i^n\mu'(\calM (L_i))$. Thus, we conclude that 
$$\PuMon(L) \cong \left(\prod_i^n\mu'(\calM (L_i))\right) \rtimes \calP_L.$$
\end{proof}

\begin{corollary}\label{cor:refined formula when there is no torsion}
    If the Khovanov homology~$\Kh(L_i)$ is torsion-free for each piece~$L_i$ of a split link~$L$, then 
   ~$$ \PuMon(L)\cong \left(\prod _i \PuMon(L_i)\right)\rtimes \calP_L.$$
\end{corollary}

\begin{proof}
    In this case,~$\Kh(L)$ is the tensor product of the~$\Kh(L_i)$, whence~$\mu'(\calM(L_i)) = \PuMon(L_i)$, and the result is now immediate from Theorem \ref{thm:monodromy_group_split_links}.
\end{proof}

It is conjectured \cite{ShumakovitchTorsion} that the only prime, non-split links with torsion-free Khovanov homology are the unknot and the Hopf link. Hence, Theorem~\ref{thm:monodromy_group_hopf-unknots} is the only expected application of Corollary~\ref{cor:refined formula when there is no torsion}. 

\subsection{Applications}
We now apply Theorem \ref{thm:monodromy_group_split_links} to an infinite class of split links.
We begin with two lemmas that deal with the monodromy group of the single link pieces: the unknot~$U$ and the Hopf link~$H$.

\begin{lemma}\label{lem:unknot}
    The group~$\PuMon(U)$ is isomorphic to~$\mathbb{Z}_2$.
\end{lemma}
\begin{proof}
    The motion group of the unknot~$U$ is generated by the flip map, which turns the unknot over by~$\pi$ \cite{Dahm62, Goldsmith}. We explicitly show that the corresponding map in~$\PuMon(U)$ is an order-$2$ element. Consider the movie in Figure~\ref{fig:Uflip_movie} representing the flip motion, where~$\text{R1}$ means Reidemeister move I.
    \begin{figure}[htb]
        \centering
        \begin{tikzpicture}
        \node[anchor=south west,inner sep=0] at (0,0){\includegraphics[width=6cm]{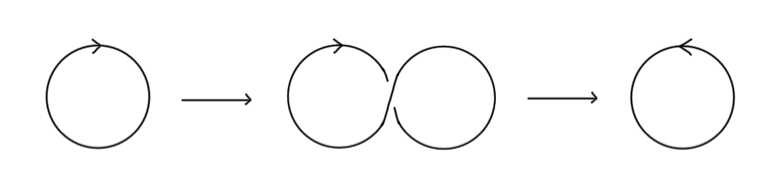}};
        \node at (1.7,1.1) {R1};
        \node at (4.4,1.1) {R1};
        \end{tikzpicture}
        \caption{A movie for the flip motion.}
        \label{fig:Uflip_movie}
    \end{figure}
    
    The corresponding sequence of cube of resolutions is drawn in Figure~\ref{fig:Uflip_res}. Using the tables from \cite{Hayden_Sundberg}, the corresponding map on homology is given by the order-$2$ matrix~$\begin{pmatrix}
           1 & 0 \\
           0 & -1 
       \end{pmatrix}$.
    \begin{figure}[htb]
        \centering
        \begin{tikzpicture}
        \node[anchor=south west,inner sep=0] at (0,0){\includegraphics[width=4cm]{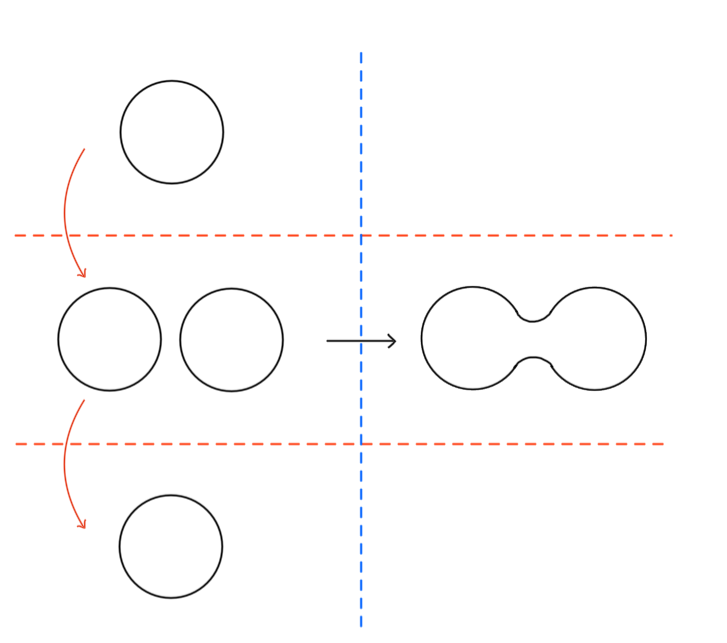}};
        \end{tikzpicture}
        \caption{Sequence of resolutions from the movie in Figure~\ref{fig:Uflip_movie}.}
        \label{fig:Uflip_res}
    \end{figure}
\end{proof}

\begin{lemma}\label{lem:hopf_link}
    Let~$H$ be a Hopf link in~$\mathbb{R}^3$. Then~$$\PuMon(H) \cong \mathbb{Z}_2 \times \mathbb{Z}_2.$$ 
\end{lemma}
\begin{proof}
   We do the computations for the negative Hopf link; those for the positive Hopf link are analogous. The motion group~$\calM(H)$ is isomorphic to the quaternionic group~$Q_8$ \cite{DamianiKamadaRingHTrivialLinks19, Boyd-Bregman-Hopf}. In particular, if one takes the embedding in Figure \ref{figure:hopf_link} as the basepoint, the generators are given by the~$\pi$ rotation along the axis~$\alpha$ (the~$x$-axis), the~$\pi$ rotation along the axis~$\beta$ ($\{z = -y\} \cap \{x = 0\}$), and their composition.
    \begin{figure}[htb]
    \centering
    \begin{tikzpicture}
    \node[anchor=south west,inner sep=0] at (0,0){\includegraphics[width=7cm]{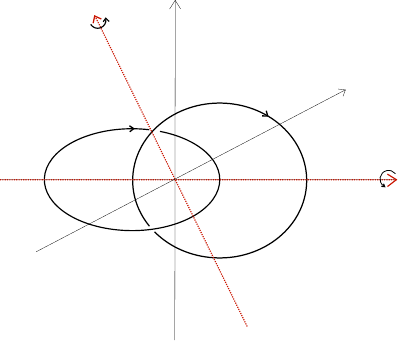}};
    \node at (7.2,3) {$\alpha$};
    \node at (2,6){$\beta$};
    \end{tikzpicture}
    \caption{A schematic of the Hopf link, with the two axis of rotation~$\alpha$ and~$\beta$.}
    \label{figure:hopf_link}
    \end{figure}

   By Theorem \ref{thm:motions_to_monodromies} we can compute~$\PuMon(H)$ by looking at the image of the generators of~$\calM(H)$ under the Khovanov motion picture map~$\mu$. Consider the oriented embedding represented in Figure \ref{figure:hopf_link} and let~$D$ be the diagram obtained by taking the projection on the plane orthogonal to~$\beta$. The diagram collection of~$D$ is the set~$\Dcal=\{D, D'\}$ with elements as in Figure~\ref{fig:Hopf_diagram_collection}.
   \begin{figure}[htb]
    \centering
    \begin{tikzpicture}[scale = 2]
    \node[anchor=south west,inner sep=0] at (0.5,0){\includegraphics[width = 2cm]{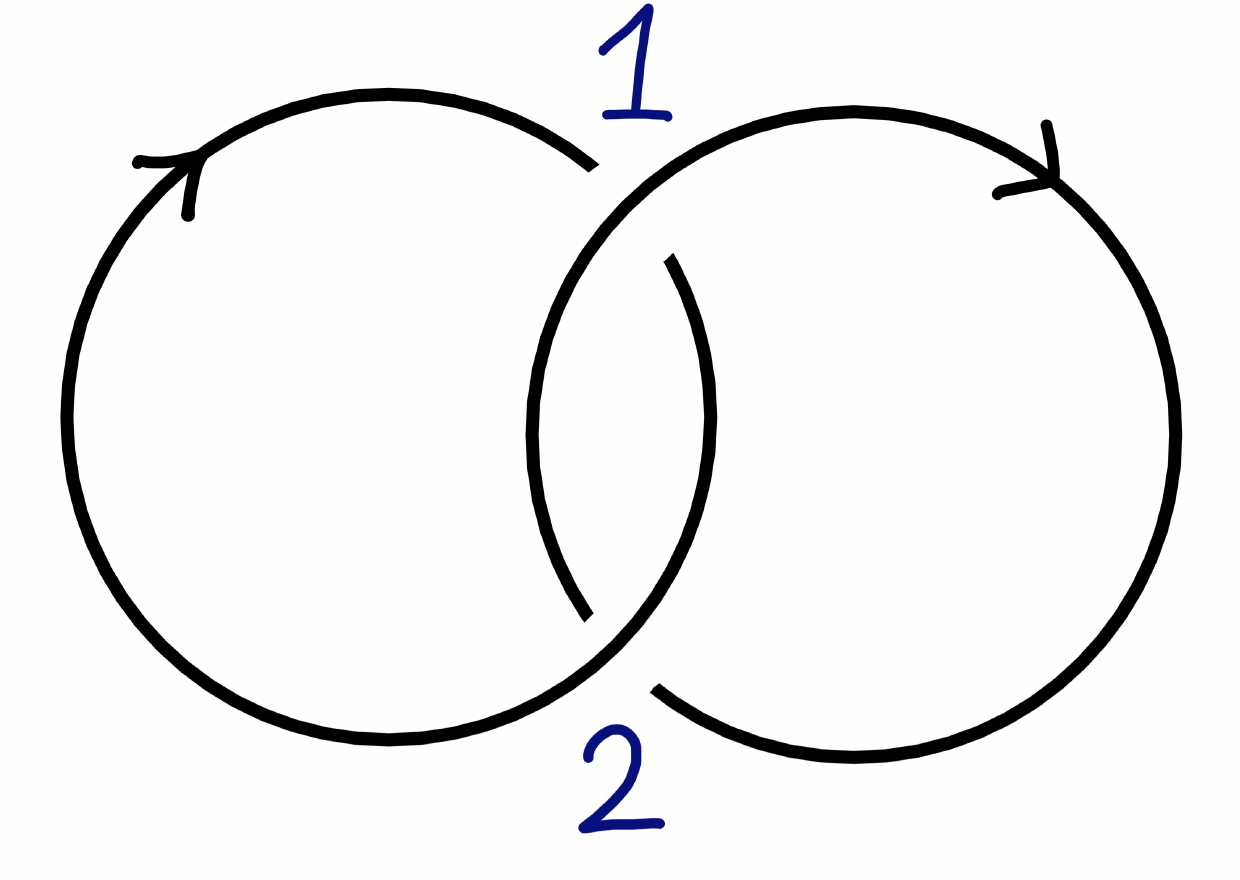}};
    \node at (0,0.7){$\Dcal= \{$};
    \node at (2.8,0.7) {,};
    \node[anchor=south west,inner sep=0] at (3,0){\includegraphics[width=2cm]{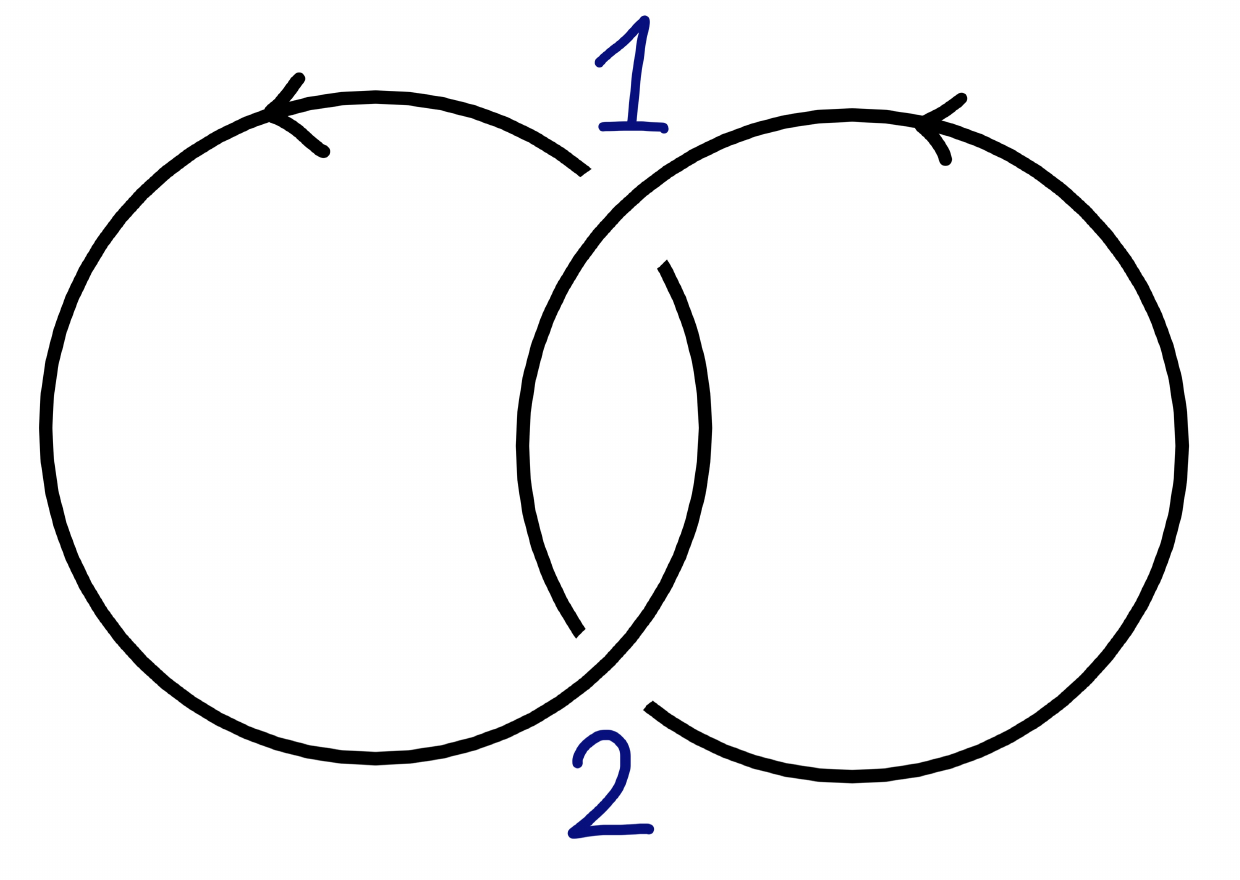}};
    \node at (5.2,0.7){$\}$};
    \end{tikzpicture}
    \caption{The diagram collection~$\Dcal$.}
    \label{fig:Hopf_diagram_collection}
    \end{figure}
   
   The Khovanov homology of the negative Hopf link is 
  ~$$\Kh(H)= \Z_{(-2,-6)}\oplus\Z_{(-2, -4)}\oplus\Z_{(0, -2)}\oplus\Z_{(0, 0)},$$ 
   where~$\Z_{(h,q)}$ is a copy of~$\Z$ in bigrading~$(h,q)$. Hence, the generating motions induce the following morphisms.
   \begin{itemize}
       \item The~$\pi$ rotation along the axis~$\beta$ corresponds to a planar isotopy on the plane of the projection exchanging the two components of the Hopf link, sending~$D$ to itself but with crossings swapped. Therefore the corresponding automorphism swaps the circles in the~$00$-resolution, and is represented by the matrix 
      ~$$\begin{pmatrix}
           1 & 0 & 0 & 0 \\
           0 & -1 & 0 & 0 \\
           0 & 0 & 1 & 0 \\
           0 & 0 & 0 & 1
       \end{pmatrix}.$$
       \item The~$\pi$ rotation along the axis~$\alpha$ corresponds to a flip of our diagram, presented by the movie in Figure~\ref{fig:flip_movie} with first still~$D$ and last still~$D'$.
       \begin{figure}[htb]
        \centering
        \begin{tikzpicture}
        \node[anchor=south west,inner sep=0] at (0.5,0){\includegraphics[width=0.62\linewidth]{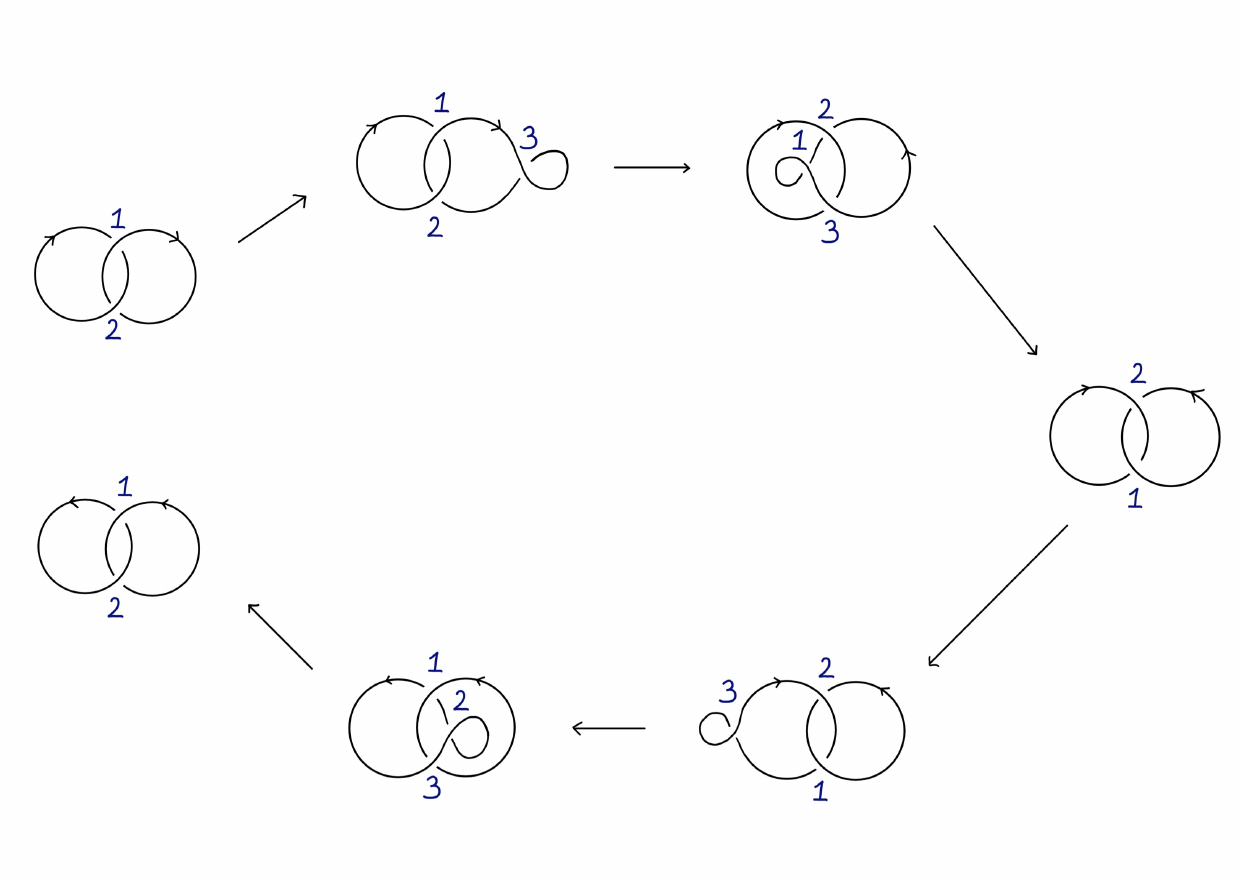}};
        \node at (2.4,5.5){R1};
        \node at (5.8,6){R3};
        \node at (8.75,5){R1};
        \node at (8.9,2.2){R1};
        \node at (5.55,1){R3};
        \node at (2.45,1.8){R1};
        \end{tikzpicture}
        \caption{The movie for the flip map.}
        \label{fig:flip_movie}
        \end{figure}
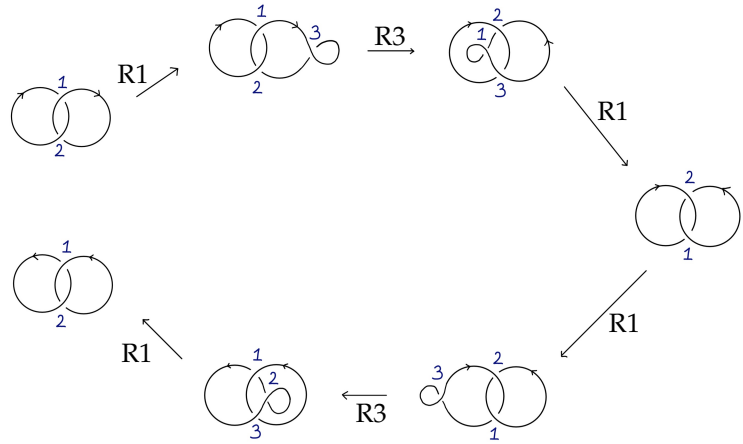
           
        Consider the map induced by the first half of the movie (that is, up to the rightmost diagram in Figure \ref{fig:flip_movie}). The corresponding sequence of cube of resolutions is drawn in Figure~\ref{fig:resolutions_hopf}, with each column corresponding to a fixed homological degree.  
        \begin{figure}[htb]
               \centering
               \includegraphics[width = 0.51\linewidth]{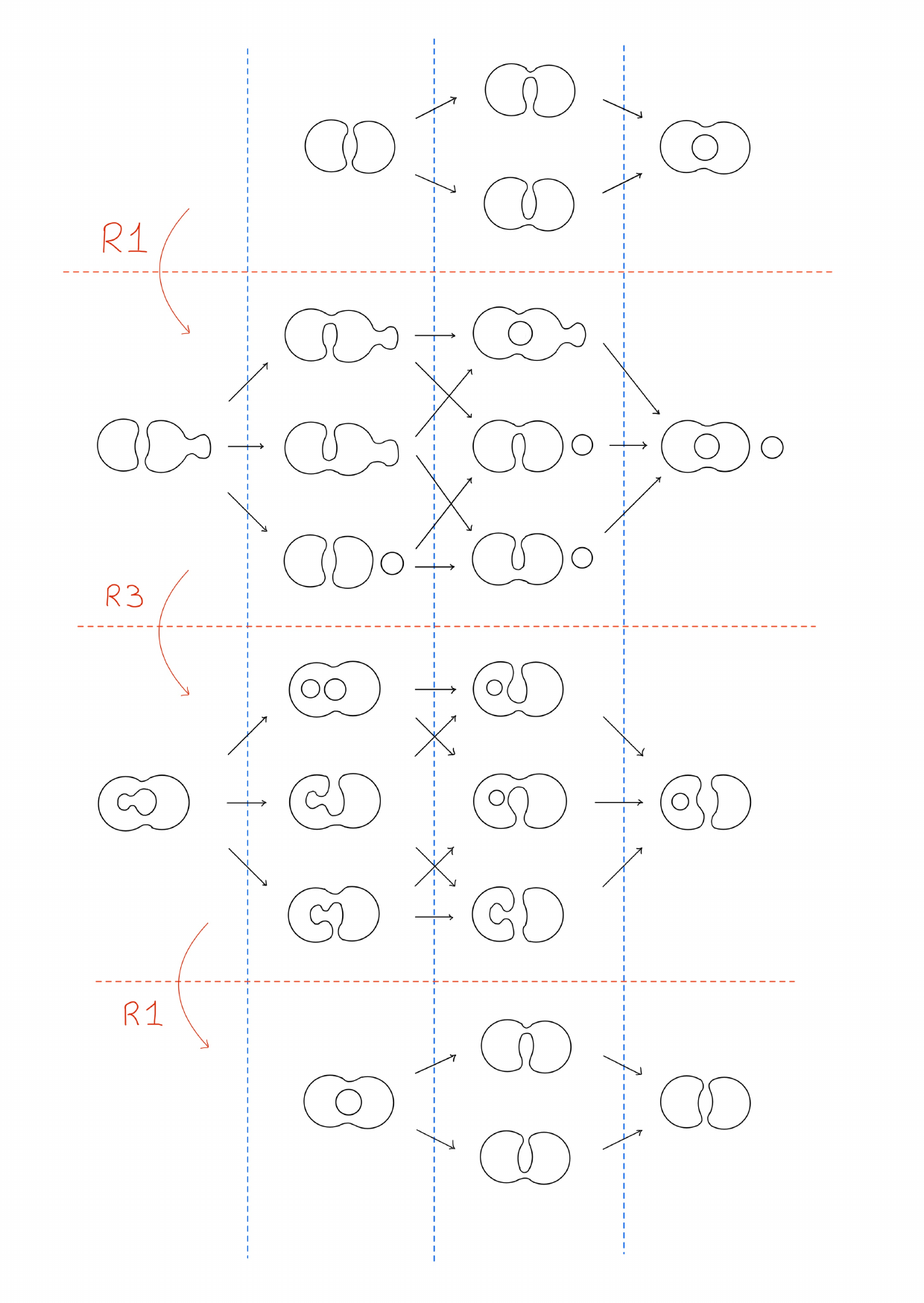}
               \caption{The cubes of resolutions associated to the diagrams in the first half of the movie.}
               \label{fig:resolutions_hopf}
        \end{figure}
        Using the tables from \cite{Hayden_Sundberg}, paying careful attention to the ordering of the crossings, we obtain the matrix 
       ~$$\begin{pmatrix}
            1 & 0 & 0 & 0 \\
            0 & 1 & 0 & 0 \\
            0 & 0 & -1 & 0 \\
            0 & 0 & 0 & -1
        \end{pmatrix}.$$
        Similarly, the map associated to the second part of the movie has matrix
       ~$$\begin{pmatrix}
            1 & 0 & 0 & 0 \\
            0 & -1 & 0 & 0 \\
            0 & 0 & -1 & 0 \\
            0 & 0 & 0 & 1
        \end{pmatrix}.$$
        Hence, the~$\pi$ rotation along the axis~$\alpha$ induces in homology the map represented by the matrix
       ~$$\begin{pmatrix}
             1 & 0 & 0 & 0 \\
             0 & -1 & 0 & 0 \\
             0 & 0 & -1 & 0 \\
             0 & 0 & 0 & 1
        \end{pmatrix}
        \begin{pmatrix}
               1 & 0 & 0 & 0 \\
               0 & 1 & 0 & 0 \\
               0 & 0 & -1 & 0 \\
               0 & 0 & 0 & -1
        \end{pmatrix}=
        \begin{pmatrix}
               1 & 0 & 0 & 0 \\
               0 & -1 & 0 & 0 \\
               0 & 0 & 1 & 0 \\
               0 & 0 & 0 & -1
        \end{pmatrix}.$$
    \end{itemize}

    We conclude that~$\PuMon(H) \cong \mathbb{Z}_2 \times \mathbb{Z}_2$.

\end{proof}

\begin{thm:hopf}
    Let~$L$ be a split link composed of~$n$ unknots and~$m$ Hopf links. Then we have 
   ~$$\PuMon(L)\cong \left(\Z_2^n \times (\Z_2\times \Z_2)^m\right) \rtimes (S_n\times S_m).$$
\end{thm:hopf}

\begin{proof}
    Combine Corollary \ref{cor:refined formula when there is no torsion}, Lemma \ref{lem:unknot}, and Lemma \ref{lem:hopf_link}.
\end{proof}

\bibliographystyle{alpha}
\bibliography{bibliography.bib}

@article{Jacobsson,
 author = {Jacobsson, Magnus},
 title = {An invariant of link cobordisms from {Khovanov} homology},
 fjournal = {Algebraic \& Geometric Topology},
 journal = {Algebr. Geom. Topol.},
 issn = {1472-2747},
 volume = {4},
 pages = {1211--1251},
 year = {2004},
 language = {English},
 doi = {10.2140/agt.2004.4.1211},
 url = {https://eudml.org/doc/125145},
 zbMATH = {2139857},
 Zbl = {1072.57018}
}

@article{Goldsmith,
 author = {Goldsmith, Deborah L.},
 title = {The theory of motion groups},
 fjournal = {Michigan Mathematical Journal},
 journal = {Mich. Math. J.},
 issn = {0026-2285},
 volume = {28},
 pages = {3--17},
 year = {1981},
 language = {English},
 doi = {10.1307/mmj/1029002454},
 zbMATH = {3723519},
 Zbl = {0462.57007}
}

@article{Hayden_Sundberg,
 author = {Hayden, Kyle and Sundberg, Isaac},
 title = {Khovanov homology and exotic surfaces in the 4-ball},
 fjournal = {Journal f{\"u}r die Reine und Angewandte Mathematik},
 journal = {J. Reine Angew. Math.},
 issn = {0075-4102},
 volume = {809},
 pages = {217--246},
 year = {2024},
 language = {English},
 doi = {10.1515/crelle-2024-0001},
 zbMATH = {7829702},
 Zbl = {1539.57019}
}

@article{Khovanov_OG,
 author = {Khovanov, Mikhail},
 title = {A categorification of the {Jones} polynomial},
 fjournal = {Duke Mathematical Journal},
 journal = {Duke Math. J.},
 issn = {0012-7094},
 volume = {101},
 number = {3},
 pages = {359--426},
 year = {2000},
 language = {English},
 doi = {10.1215/S0012-7094-00-10131-7},
 zbMATH = {1441958},
 Zbl = {0960.57005}
}

@phdthesis{Dahm62,
    author = {David M. Dahm},
    title = {A generalisation of braid theory},
    school = {Princeton University},
    year = {1962},
}

@mastersthesis{Collari,
    author = {Carlo Collari},
    title = {The functoriality of {Khovanov} Homology and the monodromy of knots},
    school = {University of Pisa},
    year = {2013},
}

@article {WattenbergSmoothMotionsUnlink72,
    AUTHOR = {Wattenberg, Frank},
     TITLE = {Differentiable motions of unknotted, unlinked circles in
              {$3$}-space},
   JOURNAL = {Math. Scand.},
  FJOURNAL = {Mathematica Scandinavica},
    VOLUME = {30},
      YEAR = {1972},
     PAGES = {107--135},
      ISSN = {0025-5521,1903-1807},
   MRCLASS = {57D40 (57E05)},
  MRNUMBER = {334243},
MRREVIEWER = {Wilbur\ Whitten},
       DOI = {10.7146/math.scand.a-11068},
       URL = {https://doi.org/10.7146/math.scand.a-11068},
}

@article{Boyd-Bregman,
 author = {Boyd, Rachael and Bregman, Corey},
 title = {Embedding spaces of split links},
 fjournal = {Advances in Mathematics},
 journal = {Adv. Math.},
 issn = {0001-8708},
 volume = {470},
 pages = {41},
 note = {Id/No 110235},
 year = {2025},
 language = {English},
 doi = {10.1016/j.aim.2025.110235},
 zbMATH = {8028625},
 Zbl = {1568.58001}
}

@article{Goldsmith_torus,
 author = {Goldsmith, D. L.},
 title = {Motion of links in the 3-sphere},
 fjournal = {Mathematica Scandinavica},
 journal = {Math. Scand.},
 issn = {0025-5521},
 volume = {50},
 pages = {167--205},
 year = {1982},
 language = {English},
 doi = {10.7146/math.scand.a-11953},
 url = {https://eudml.org/doc/166791},
 zbMATH = {3748049},
 Zbl = {0476.57004}
}

@article{Gujral-Levine,
 author = {Gujral, Onkar Singh and Levine, Adam Simon},
 title = {Khovanov homology and cobordisms between split links},
 fjournal = {Journal of Topology},
 journal = {J. Topol.},
 issn = {1753-8416},
 volume = {15},
 number = {3},
 pages = {973--1016},
 year = {2022},
 language = {English},
 doi = {10.1112/topo.12244},
 zbMATH = {7738176},
 Zbl = {1532.57003}
}

@misc{Boyd-Bregman-Hopf,
 author = {Rachael Boyd and Corey Bregman},
 title = {The embedding space of a {Hopf} link},
 year = {2025},
 howpublished = {preprint, {arXiv}:2504.21806},
 url = {https://arxiv.org/abs/2504.21806},
 arXiv = {arXiv:2504.21806}
}

@article{Rasmussen,
 author = {Rasmussen, Jacob},
 title = {Khovanov homology and the slice genus},
 fjournal = {Inventiones Mathematicae},
 journal = {Invent. Math.},
 issn = {0020-9910},
 volume = {182},
 number = {2},
 pages = {419--447},
 year = {2010},
 language = {English},
 doi = {10.1007/s00222-010-0275-6},
 zbMATH = {5818344},
 Zbl = {1211.57009}
}

@misc{Ren-Willis,
 author = {Qiuyu Ren and Michael Willis},
 title = {Khovanov homology and exotic {$4$}-manifolds},
 year = {2025},
 howpublished = {Preprint, {arXiv}:2402.10452 [math.{GT}] (2025)},
 url = {https://arxiv.org/abs/2402.10452},
 arXiv = {arXiv:2402.10452}
}

@article{Piccirillo,
 author = {Piccirillo, Lisa},
 title = {The {Conway} knot is not slice},
 fjournal = {Annals of Mathematics. Second Series},
 journal = {Ann. Math. (2)},
 issn = {0003-486X},
 volume = {191},
 number = {2},
 pages = {581--591},
 year = {2020},
 language = {English},
 doi = {10.4007/annals.2020.191.2.5},
 zbMATH = {7168645},
 Zbl = {1471.57011}
}

@article{Carter-Saito-1,
 author = {Carter, J. Scott and Saito, Masahico},
 title = {Reidemeister moves for surface isotopies and their interpretation as moves to movies},
 fjournal = {Journal of Knot Theory and its Ramifications},
 journal = {J. Knot Theory Ramifications},
 issn = {0218-2165},
 volume = {2},
 number = {3},
 pages = {251--284},
 year = {1993},
 language = {English},
 doi = {10.1142/S0218216593000167},
 zbMATH = {488130},
 Zbl = {0808.57020}
}

@book{MacLane,
 author = {Mac Lane, Saunders},
 title = {Homology.},
 fseries = {Grundlehren der Mathematischen Wissenschaften},
 series = {Grundlehren Math. Wiss.},
 issn = {0072-7830},
 volume = {114},
 year = {1963},
 publisher = {Springer, Cham},
 language = {German},
 zbMATH = {3216237},
 Zbl = {0133.26502}
}

@article {FouxeRabinovitchAutsFreeProductsI40,
    AUTHOR = {Fouxe-Rabinovitch, D. I.},
     TITLE = {\"{U}ber die {A}utomorphismengruppen der freien {P}rodukte. {I}},
   JOURNAL = {Rec. Math. [Mat. Sbornik] N.S.},
  FJOURNAL = {Rec. Math. [Mat. Sbornik] N.S.},
    VOLUME = {8/50},
      YEAR = {1940},
     PAGES = {265--276},
   MRCLASS = {20.0X},
  MRNUMBER = {3413},
MRREVIEWER = {L.\ Zippin},
}

@article {FouxeRabinovitchAutsFreeProductsII41,
    AUTHOR = {Fouxe-Rabinovitch, D. I.},
     TITLE = {\"{U}ber die {A}utomorphismengruppen der freien {P}rodukte.
              {II}},
   JOURNAL = {Rec. Math. [Mat. Sbornik] N.S.},
  FJOURNAL = {Rec. Math. [Mat. Sbornik] N.S.},
    VOLUME = {9/51},
      YEAR = {1941},
     PAGES = {183--220},
   MRCLASS = {20.0X},
  MRNUMBER = {4625},
}

@article {DamianiKamadaRingHTrivialLinks19,
    AUTHOR = {Damiani, Celeste and Kamada, Seiichi},
     TITLE = {On the group of ring motions of an {H}-trivial link},
   JOURNAL = {Topology Appl.},
  FJOURNAL = {Topology and its Applications},
    VOLUME = {264},
      YEAR = {2019},
     PAGES = {51--65},
      ISSN = {0166-8641,1879-3207},
   MRCLASS = {57M07 (20F36 57M25)},
  MRNUMBER = {3974715},
MRREVIEWER = {J.\ S.\ Birman},
       DOI = {10.1016/j.topol.2019.06.004},
       URL = {https://doi.org/10.1016/j.topol.2019.06.004},
}

@article {BrendleHatcherRingsWickets13,
    AUTHOR = {Brendle, Tara E. and Hatcher, Allen},
     TITLE = {Configuration spaces of rings and wickets},
   JOURNAL = {Comment. Math. Helv.},
  FJOURNAL = {Commentarii Mathematici Helvetici. A Journal of the Swiss
              Mathematical Society},
    VOLUME = {88},
      YEAR = {2013},
    NUMBER = {1},
     PAGES = {131--162},
      ISSN = {0010-2571,1420-8946},
   MRCLASS = {20F36 (57M07)},
  MRNUMBER = {3008915},
MRREVIEWER = {Juan\ Gonz\'alez-Meneses},
       DOI = {10.4171/CMH/280},
       URL = {https://doi.org/10.4171/CMH/280},
}

@article {BellingeriBodinBraidGroupNecklace16,
    AUTHOR = {Bellingeri, Paolo and Bodin, Arnaud},
     TITLE = {The braid group of a necklace},
   JOURNAL = {Math. Z.},
  FJOURNAL = {Mathematische Zeitschrift},
    VOLUME = {283},
      YEAR = {2016},
    NUMBER = {3-4},
     PAGES = {995--1010},
      ISSN = {0025-5874,1432-1823},
   MRCLASS = {20F34 (20F36 57M25)},
  MRNUMBER = {3519992},
MRREVIEWER = {Bruno\ P.\ Zimmermann},
       DOI = {10.1007/s00209-016-1630-0},
       URL = {https://doi.org/10.1007/s00209-016-1630-0},
}

@article {ShumakovitchTorsion,
    AUTHOR = {Shumakovitch, Alexander N.},
     TITLE = {Torsion of {K}hovanov homology},
   JOURNAL = {Fund. Math.},
  FJOURNAL = {Fundamenta Mathematicae},
    VOLUME = {225},
      YEAR = {2014},
    NUMBER = {1},
     PAGES = {343--364},
      ISSN = {0016-2736,1730-6329},
   MRCLASS = {57M27 (57M25)},
  MRNUMBER = {3205577},
MRREVIEWER = {Pedro\ Vaz},
       DOI = {10.4064/fm225-1-16},
       URL = {https://doi.org/10.4064/fm225-1-16},
}

\end{document}